\documentclass{article} 
\usepackage{graphicx}
\usepackage{amsmath, amssymb, amsthm} 
\usepackage{verbatim}                  %ƒRƒƒ"ƒg
\usepackage[alphabetic]{amsrefs}       %Reference
\usepackage{color}
\theoremstyle{plain} 
\newtheorem{theorem}{Theorem}
\newtheorem{lemma}[theorem]{Lemma}
\newtheorem{corollary}[theorem]{Corollary}
\newtheorem{proposition}[theorem]{Proposition}

\newtheorem*{theoremTPhitoT}{Theorem \ref{thm:TPhitoT}}
\newtheorem*{theoremTtoTPhi}{Theorem \ref{thm:TtoTPhi}}
\newtheorem*{theoremTtoFPhi}{Theorem \ref{thm:TtoFPhi}}
\newtheorem*{theoremnotFPhi}{Theorem \ref{thm:notFPhi}}

\theoremstyle{definition} 
\newtheorem{definition}[theorem]{Definition}

\newtheorem*{acknowledgements}{Acknowledgements}
\newcommand{\sign}{\operatorname{sign}}

\newcommand{\supp}{\operatorname{supp}}
\title{Property $(T_{L^{\Phi}})$ and property $(F_{L^{\Phi}})$ for Orlicz spaces $L^{\Phi}$}
\author{Mamoru Tanaka\footnote{This work was supported by World Premier International Research Center Initiative (WPI), MEXT, Japan}}
\date{} 
\begin{document}

\maketitle

\begin{abstract}
An Orlicz space $L^{\Phi}(\Omega)$ is a Banach function space defined by using a Young function $\Phi$, which generalizes the $L^p$ spaces. 
We show that, for a reflexive Orlicz space $L^{\Phi}([0,1])$, a locally compact second countable group has Kazhdan's property $(T)$ if and only if it has property $(T_{L^{\Phi}([0,1])})$, which is a generalization of Kazhdan's property $(T)$ for linear isometric representations on $L^{\Phi}([0,1])$. 
We also prove that, for a Banach space $B$ whose modulus of convexity is sufficiently large, if a locally compact second countable group has Kazhdan's property $(T)$, then it has property $(F_{B})$, which is a fixed point property for affine isometric actions on $B$. 
Moreover, we see that, for an Orlicz sequence space $\ell^{\Phi\Psi}$ such that the Young function $\Psi$ sufficiently rapidly increases near $0$,  hyperbolic groups (with Kazhdan's property $(T)$) don't have property $(F_{\ell^{\Phi\Psi}})$. 
These results are generalizations of the results for $L^p$-spaces. 
\\


\end{abstract}

%%%%%%%%%%%%%%%%%%%%%%%%%%%%%%%%%%%%%%%%%%%%%%%%%%%%%%%%%%%%%%%%%%%%%%%%%%%%%%%%%%

\section{Introduction} 

Property $(T)$ is known as a rigidity property of topological groups with respect to the irreducible unitary representations.  
In \cite{MR2316269}, they generalize Kazhdan's property $(T)$ for linear isometric representations on Banach spaces:  
Let $G$ be a topological group and $(B,\|\ \|)$ a Banach space. 
A {\it linear isometric $G$-representation} on $B$ is a continuous homomorphism $\rho :G\to O(B) $, 
where $O(B)$ denotes the group of all invertible linear isometries $B \to B$, and continuous means the action map $G \times B \to B$ is %jointly 
continuous. 
We say that a linear isometric $G$-representation $\rho$ {\it almost has invariant vectors} if for all compact subsets $K\subset G$
$$ \inf_{v\in B;\| v \|=1 } \max_{g \in K} \|\rho (g )v -v\| = 0 .  $$
Denote by $B^{\rho(G)}$ the closed subspace of $G$-fixed vectors in $B$. 
Then the $G$-representation $\rho$ descends to a linear isometric $G$-representation $\tilde \rho$ on $B/B^{\rho(G)}$. 

\begin{definition}[\cite{MR2316269}]
Let $B$ be a Banach space. 
A topological group $G$ is said to have {\it property $(T_B)$} 
if for any linear isometric $G$-representation $\rho : G\to O(B) $,
 the quotient $G$-representation $\tilde \rho : G\to O(B/B^{\rho(G)}) $ does not almost have invariant vectors. 
\end{definition}
For a Hilbert space $H$, Kazhdan's property $(T)$ is equivalent to property $(T_H)$. 
For a locally compact second countable group, Delorme \cite{MR0578893} and Guichardet \cite{MR0340464} proved that Kazhdan's property $(T)$ is equivalent to Serre's property $(FH)$, that is, every affine isometric action on a real Hilbert space has a fixed point. 
\begin{definition}[\cite{MR2316269}]
Let $B$ be a Banach space. 
A topological group $G$ is said to have {\it property $(F_B)$} 
if every affine isometric $G$-action on $B$ has a fixed point.  
\end{definition}

In \cite{MR2316269}, 
they proved the following:  
\begin{theorem}[\cite{MR2316269}]
Let $G$ be a locally compact second countable group, $B$ a Banach space. %, and $(\Omega, \mu)$ a standard Borel space. 
\begin{enumerate}
\item If $G$ has property $(F_B)$, then $G$ has property $(T_B)$. 
\item If $G$ has property $(T_{L^p([0,1])})$ for some $1\le p<\infty$, then $G$ has Kazhdan's property $(T)$. 
\item If $G$ has Kazhdan's property $(T)$, then $G$ has property $(T_{L^p(\mu)})$ for any $\sigma$-finite measure $\mu $ and any $1\le p<\infty$. 
\item If $G$ has Kazhdan's property $(T)$, then there  exists a constant $\epsilon(G)> 0$ such that $G$ has property $(F_{L^p(\mu)})$ for any $\sigma$-finite measure $\mu $ and any $1\le p<2+\epsilon(G)$. 
\end{enumerate}
\end{theorem}

On the other hand, Yu proved 
\begin{theorem}[\cite{MR2221161}]
If $\Gamma$ is a hyperbolic group, then there exists $2 \le p(\Gamma ) < \infty$ such that $\Gamma $ admits a proper affine isometric action on an $\ell^p$-space for $p\ge p(\Gamma)$.
\end{theorem}
There are hyperbolic groups which have Kazhdan's property $(T)$, for example, the cocompact lattices of $Sp_{n,1}(\mathbb{R})$ $(n\ge 2)$. 
Hence property $(F_{\ell^p})$ and property $(T_{\ell^p})$ are not equivalent for sufficiently large $p\ge 2$.

T. Yokota asked us whether results about isometric group actions on $L^p$-spaces is true for Orlicz spaces $L^\Phi$ under appropriate conditions. 
Here an Orlicz space is a generalization of $L^p$-spaces, which is defined in section \ref{sec:Orlicz}. 
We prove the following:   

\begin{theorem}\label{thm:TPhitoT}
Let $G$ be a locally compact second countable group,  
$\Phi $ a Young function with $0< \Phi(t) < \infty$ for all $t>0$, 
and $\mathbb{K}=\mathbb{R}$ or $\mathbb{C}$. 
If $G$ has property $(T_{L^\Phi([0,1], \mathbb{K})})$, then it has Kazhdan's property $(T)$. 
\end{theorem}
$L^p$ spaces ($1\le p<\infty$) are Orlicz spaces satisfying the assumption of Theorem \ref{thm:TPhitoT}. 

\begin{theorem}\label{thm:TtoTPhi}
Let $G$ be a locally compact second countable group, and $\Phi$ an N-function such that $\Phi \in \Delta_2^\Omega \cap \nabla_2^\Omega$, and 
\begin{center}
$(1)$ $\Omega=[0,1]$ and $\mathbb{K}=\mathbb{R}$,  \ \ or \ \ \
$(2)$ $\Omega=[0,1]$ and $\mathbb{K}=\mathbb{C}$,  \ \ or \ \ \ 
$(3)$ $\Omega=\mathbb{N}$ and $\mathbb{K}=\mathbb{C}$. %and the standard unit vectors form an 1-symmetric basis in $L^\Phi(\Omega,\mathbb{K})$,  
\end{center}
If $G$ has Kazhdan's property $(T)$, then it has property $(T_{L^\Phi(\Omega,\mathbb{K})})$ with respect to gauge norm. 
\end{theorem}
For an N-function $\Phi$, $\Phi \in \Delta_2^\Omega \cap \nabla_2^\Omega$ if and only if $L^\Phi(\Omega, \mathbb{K})$ is reflexive. 
Hence $L^p$ spaces ($1 < p < \infty$) are Orlicz spaces satisfying the assumption of Theorem \ref{thm:TtoTPhi}. 

\begin{corollary}
Let $G$ be a locally compact second countable group.  
For a reflexive Orlicz space $L^\Phi([0,1])$, 
$G$ has has Kazhdan's property $(T)$ if and only if it has property $(T_{L^\Phi([0,1], \mathbb{K})})$. 
\end{corollary}

A Banach space $B$ is said to be {\it uniformly convex} 
if for every $\epsilon >0$ the modulus of convexity 
$$ \delta_B(\epsilon) = \inf \left\{ 1-\left\|\frac{u+v}{2} \right\| :  \|u\|= \|v\|= 1, \|u-v\|\ge \epsilon \right\}$$
is positive. 
A Banach space $B$ is said to be {\it uniformly smooth} 
if the modulus of smoothness 
$$ \rho_B(\tau ) = \inf \left\{ \frac{\|u+v\| + \|u-v\|}{2} -1 : \|u\|= 1, \|v\|= \tau \right\}, $$
satisfies $\lim_{\tau \to 0}\frac{\rho(\tau)}{\tau}=0$. 
By \cite{MR1501880}, the modulus of convexity is calculated as $\delta_{L^p}(\epsilon)=1-(1-(\frac{\epsilon}{2})^p)^{\frac{1}{p}}$ for $p\ge 2$.  
Uniformly convex (or uniformly smooth) Banach spaces are reflexive.

\begin{theorem} \label{thm:TtoFPhi}
If $G$ has Kazhdan's property $(T)$, then there exists a constant $\epsilon (G)>0$ such that $G$ has property $(F_{B})$ for every real $($or complex$)$ Banach spaces $B$ with $\delta_B(t) \ge \delta_{L^{2+\epsilon(G)}}(t)$ for all $0<t<2$ $($or with $\rho_B(t) \le \rho_{L^{2+\epsilon(G)}}(t)$ for all $t>0)$. 
\end{theorem}
$L^p$ spaces ($2 \le p \le 2 + \epsilon (G)$) are Orlicz spaces satisfying the assumption of Theorem \ref{thm:TtoFPhi}.

\begin{theorem}\label{thm:notFPhi}
Let $\Gamma $ be a hyperbolic group and $\mathbb{K}=\mathbb{R}$ or $\mathbb{C}$. 
Then there exists $2 \le p(\Gamma) < \infty$ such that,  
for any N-functions $\Phi$ and $\Psi$ satisfying 
\begin{itemize} 
%\item $0<\Phi(t)<\infty$ and $0<\Phi^*(t)<\infty$ for all $t>0$,  
\item $\Phi \in \Delta_2^{\mathbb{N}}$ $($this is equivalent to $\ell^\Phi(\Gamma, \mathbb{K})$ is separable$)$ and    
\item there is a constant $C>0$ and $t_0>0$ such that $\Psi (t) \le Ct^{p(\Gamma)}$ for all $0< t \le t_0$, 
\end{itemize}
the group $\Gamma $ admits a proper affine isometric action on $\ell^\Psi(\Gamma,\ell^\Phi(\Gamma, \mathbb{K}))$ with gauge norm. 
\end{theorem}

For example, $\ell^p=\ell^p(\Gamma \times \Gamma)$ with $p \ge p(G)$ are Orlicz spaces satisfying the assumption of Theorem \ref{thm:notFPhi}. 
\begin{acknowledgements}
We would like to thank T. Yokota for asking the question and encouraging us. 
\end{acknowledgements}

\section{Orlicz spaces} 
\label{sec:Orlicz}

This section refers to \cite{MR1113700} and \cite{MR1957004}.  

A function $\Phi :[0,+\infty) \to [0,+\infty]$ is said to be a {\it Young function} if it is 
\begin{enumerate}
\item convex, i.e., $\Phi(st_1+(1-s)t_2) \le s\Phi(t_1) + (1-s)\Phi(t_2)$ for all $t_1,t_2\in [0,+\infty)$ and $s\in [0,1]$;
\item $\Phi (0)=0$; 
\item $\lim_{t\to \infty} \Phi (t) = +\infty$.  
\end{enumerate}
The function ${\Phi^*} :[0,+\infty) \to [0,+\infty]$ defined by 
${\Phi^*}(s):=\sup \{st-\Phi(t) : t\ge 0\}$
is called the {\it complementary function} of $\Phi$, which is also a Young function. 
A Young function $\Phi$ is called an {\it N-function} if it is a Young function satisfying $0<\Phi(t)<\infty$ for all $t\in (0,\infty)$, and $\lim_{t\to 0}\frac{\Phi (t)}{t}=0$, $\lim_{t\to \infty }\frac{\Phi (t)}{t} = 	\infty$. 
The complementary function of an N-function is also an N-function. 
For example, the function $\Phi_p(t)= \frac{t ^p}{p}$ ($1<p<\infty$) is an N-function, and the complementary function is $\Phi_p^*(s)=\frac{s ^q}{q}$, where $\frac{1}{p}+\frac{1}{q}=1$. 
The functions 
$$\Phi_1(t) =|t| \ \ \ \ \text{ and } \ \ \ \ \Phi_\infty (t) = \begin{cases}
    0 & (t \in [-1,1]) \\
    +\infty & (\text{otherwise})
  \end{cases}$$ 
are Young functions, but they are not N-function.

Let $\Omega$ be a $\sigma $-finite measure space with a positive measure $\mu$, and $\mathbb{K}=\mathbb{R}$ or $\mathbb{C}$. 

\begin{definition}
For a Young function $\Phi$, the space 
$$ L^\Phi (\Omega,\mathbb{K}) := \left\{ f:\Omega \to \mathbb{K} \mid \text{measurable},\  \int_\Omega \Phi (a|f|) d\mu <\infty \text{ for some  }a >0\right\}/\sim $$
is called an {\it Orlicz space}, where $f\sim g$ means $f=g$ $\mu$-a.e.. 
For $f\in L^\Phi$, we define 
$$\|f\|_{(\Phi)} := \inf \left\{ b>0: \int_\Omega \Phi \left(\frac{|f|}{b}\right) d\mu \le 1\right\} .  $$
The $\|\ \|_{(\Phi)}$ is a norm on $L^\Phi$, which is called the {\it gauge norm} (or the {\it Luxemburg-Nakano norm}). 
For $f\in L^\Phi$, we define 
$$\|f \|_\Phi := \sup \left\{ \int_\Omega |f\psi| d\mu : \int_\Omega \Phi^* \left(|\psi|\right) d\mu \le 1 \right\}. $$
The $\| \cdot \|_\Phi$ is also a norm on $ L^\Phi$, which is called the {\it Orlicz norm}. 
The norm spaces $(L^{\Phi}, \|\ \|_{(\Phi)})$ and $(L^{\Phi}, \|\cdot \|_\Phi)$ are Banach spaces. 
\end{definition}
Since $(L^{\Phi_p}, \|\ \|_{(\Phi_p)})=(L^p, \|\ \|_{L^p})$ for $1\le p\le \infty$, an Orlicz spaces with gauge norm are generalization of $L^p$ spaces. 
Gauge norm and Orlicz norm have the relation  
\begin{equation} \label{eq:gaugeorlicz}
\| f \|_{(\Phi)} \le \| f \|_\Phi \le 2 \| f\|_{(\Phi)}
\end{equation} 
for any $f\in L^\Phi$. 

We consider the following conditions for $\Omega=[0,1]$ with the Lebesgue measure $\mu$, and $\Omega=\mathbb{N}$ with the counting measure $\mu$. 
A Young function $\Phi $ is said to satisfy the {\it $\Delta_2^{\Omega}$-condition} and denoted as $\Phi \in \Delta_2^{\Omega}$
if there are $K>0$ and $t_0 > 0$ such that 
$$ \Phi(2t) \le K\Phi(t) \ \ \ \text{ for all }\ t\ge t_0  \text{ if } \Omega =[0,1]\ \  ( \text{for all }\ 0<t\le t_0 \text{ if } \Omega =\mathbb{N}).  $$
A Young function $\Phi $ is said to satisfy the {\it $\nabla_2^\Omega$-condition}  and denoted as $\Phi \in \nabla_2^{\Omega}$ 
if there are $c > 1$ and $t_0 > 0$ such that 
$$ 2c \Phi(t) \le \Phi(c t) \ \ \ \text{ for all }\  t\ge t_0 \text{ if } \Omega =[0,1]\ \  (\text{for all }\ 0<t\le t_0 \text{ if } \Omega =\mathbb{N}).  $$
For example, $\Phi_p \in \Delta_2^\Omega$ for $1\le p<\infty$ and $\Phi_p \in \nabla_2^\Omega $ for $1 < p<\infty$. 
If $\Phi \in \Delta_2^\Omega$, then the simple functions on $\Omega$ are dense in $L^\Phi$, and  
$$\int_{\Omega} \Phi\left( \frac{f}{\|f\|_{(\Phi)}}\right)d\mu=1 $$ for $f \in L^\Phi$ with $f\not=0$. 
For an N-function $\Phi$, $\Phi \in \Delta_2^\Omega \cap \nabla_2^\Omega$ if and only if $L^\Phi(\Omega, \mathbb{K})$ is reflexive. 
Note that the uniform continuity and uniform smoothness of $L^\Phi(\Omega, \mathbb{R})$ can be written by conditions for $\Phi$ and $\Phi^*$, which are strictly stronger than $\Phi \in \Delta_2^\Omega \cap \nabla_2^\Omega$. 
\begin{theorem}[\cite{MR1890178}]
For a N-function $\Phi$ and $0 < s \le 1$, let $\Phi_{(s)} $ be the inverse function of 
$ \Phi_{(s)} ^{-1}(t) = (\Phi^{-1}(t))^{1-s}t^{\frac{s}{2}} $. 
Then $\Phi_{(s)} $ is also an N-function and 
for $0<\epsilon \le 2$ 
$$ \delta_{L^{\Phi_{(s)}}(\Omega, \mathbb{R})}(\epsilon) \ge 1 - \left(1-\left(\frac{\epsilon}{2}\right)^{\frac{2}{s}} \right)^{\frac{s}{2}} = \delta_{L^{\frac{2}{s}}(\Omega, \mathbb{R})}(\epsilon) . $$
\end{theorem}
The function $\Phi_{(s)}$ is an N-function between $\Phi_{(0)} = \Phi $ and $ \Phi_{(1)} = \Phi_2$ in some sense. 
For example, let $p(s) = \frac{1}{(1-s)\frac{1}{p}+s\frac{1}{2}}$ for $1\le p<\infty$, then 
$$\Phi_{p,(s)} = \frac{t^{p(s)}}{p^{\frac{1-s}{p}p(s)}}= \frac{p(s)}{p^{\frac{1-s}{p}p(s)}} \Phi_{p(s)}. $$
Hence we can easily construct uniformly convex Orlicz spaces satisfying the assumption of Theorem \ref{thm:TtoFPhi}.

\section{Proof of Theorem \ref{thm:TPhitoT}}

\begin{theoremTPhitoT}
Let $G$ be a locally compact second countable group,  
$\Phi $ a Young function with $0< \Phi(t) < \infty$ for all $t>0$, 
and $\mathbb{K}=\mathbb{R}$ or $\mathbb{C}$. 
If $G$ has property $(T_{L^\Phi([0,1], \mathbb{K})})$, then it has Kazhdan's property $(T)$. 
\end{theoremTPhitoT}

\begin{proof}
Assume that $G$ does not have Kazhdan's property $(T)$. 
Connes and Weiss \cite{MR599455} construct a measure-preserving, ergodic $G$-action on a standard non-atomic probability space $(\Omega, \mu)$ which admits an asymptotically invariant measurable subsets $\{E_n\}_{n=1}^\infty$ such that 
\begin{eqnarray*}
\mu (E_n) = \frac{1}{2} \text{\ \ \  and \ \ } \mu (g E_n \triangle E_n)\to 0 \text{\ \ for all \ } g \in G. 
\end{eqnarray*}
As in 4.c in \cite{MR2316269}, we can take this $\{E_n\}_{n=1}^\infty$ such as the convergence $\mu (g E_n \triangle E_n)\to 0 $ is uniform on compact subsets of $G$.
Consider the linear isometric $G$-representation $\rho$ on $B=L^\Phi(\Omega, \mathbb{K}) $ defined by $\rho(g )f(x)=f(g^{-1}x)$. 
Then $B^{\rho(G)}= \mathbb{K}\chi_\Omega$, the constant functions on $\Omega$. 
Let $\tilde{B} = B/B^{\rho(G)}$. 
From the assumption, $\Phi$ has the inverse $\Phi^{-1}:[0,\infty) \to [0,\infty)$.  
Hence for $\tilde{f}_n = 2 \chi_{E_n} -\chi_\Omega  + \mathbb{K}\chi_\Omega \in \tilde{B}, $
we have 
\begin{eqnarray*}
\| \tilde{f}_n \|_{\tilde{B}} 
&=&  \inf_{a \in \mathbb{K}} \|2 \chi_{E_n} -\chi_\Omega  +  a\chi_\Omega \|_{(\Phi)} \\ 
%&=&  \inf_{a \in \mathbb{R}} \inf \left\{ b>0 \mid \int_\Omega \Phi \left(\frac{|f_n + a\chi_\Omega|}{b}\right) d\mu \le 1 \right\}\\
&=&  \inf_{a \in \mathbb{K}} \inf \left\{ b>0 \mid \int_\Omega \Phi \left(\frac{|2 \chi_{E_n} -\chi_\Omega + a\chi_\Omega|}{b}\right) d\mu \le 1 \right\}\\
&=&  \inf_{a \in \mathbb{K}} \inf \left\{ b>0 \mid \int_{E_n} \Phi \left(\frac{|1+a|}{b}\right) d\mu + \int_{\Omega -E_n} \Phi \left(\frac{| a-1|}{b}\right) d\mu \le 1 \right\}\\
&=&  \inf_{a \in \mathbb{K}} \inf \left\{ b>0 \mid \frac{1}{2}\Phi \left(\frac{|1+a|}{b}\right) + \frac{1}{2} \Phi \left(\frac{|1- a|}{b}\right) \le 1 \right\}\\
%&\ge &  \inf_{a \in \mathbb{K}} \inf \left\{ b>0 \mid \max \left\{ \Phi \left(\frac{|1+a|}{b}\right), \Phi \left(\frac{|1- a|}{b}\right)\right\} \le 2 \right\}\\
&\ge &  \inf \left\{ b>0 \mid \Phi \left(\frac{1}{b}\right) \le 2 \right\}\\
&=& \frac{1}{\Phi^{-1}(2)} > 0. 
\end{eqnarray*}
and 
\begin{eqnarray*}
\| \tilde{\rho}(g) \tilde{f}_n - \tilde{f}_n\|_{\tilde{B}}  
&=& \inf_{a\in \mathbb{K}} \| \rho(g) (2 \chi_{E_n} -\chi_\Omega) - (2 \chi_{E_n} -\chi_\Omega) + a \chi_{\Omega} \|_{(\Phi)} \\ 
&\le & \| \rho(g) (2 \chi_{E_n} -\chi_\Omega) - (2 \chi_{E_n} -\chi_\Omega)\|_{(\Phi)}  \\ 
%&=& 2\| \chi_{gE_n} - \chi_{E_n} \|_B   \\ 
&=& 2\| \chi_{gE_n \triangle E_n} \|_{(\Phi)}  \\ 
&=& 2 \inf \left\{ b>0 \mid \int_{\Omega} \Phi \left(\frac{\chi_{gE_n \triangle E_n}}{b}\right) d\mu \le 1 \right\}  \\ 
&=& 2 \inf \left\{ b>0 \mid \Phi \left(\frac{1}{b}\right) \mu(gE_n \triangle E_n) \le 1 \right\}  \\ 
&=& \frac{2}{\Phi^{-1}(\frac{1}{\mu(gE_n \triangle E_n)})} . 
\end{eqnarray*}
Define $\tilde{f}_n' = \frac{\tilde{f}_n} {\|\tilde{f}_n\|_{\tilde{B}}}$. 
Since $\Phi^{-1}(t)\to \infty$ as $t\to \infty$, we have 
\begin{eqnarray*}
\|\tilde{\rho}(g) \tilde{f}_n' - \tilde{f}_n'\|_{\tilde{B}}  
\le \frac{\| \tilde{\rho}(g) \tilde{f}_n - \tilde{f}_n\|_{\tilde{B}}  
}{\|\tilde{f}_n \|_{\tilde{B}} }
\le \frac{2\Phi^{-1}(2)}{\Phi^{-1}(\frac{1}{\mu(gE_n \triangle E_n)})} 
\to 0 
\end{eqnarray*}
as $n \to \infty$ uniformly on compact subsets of $G$.
This means $\tilde{\rho}$ almost has the invariant vectors $\{\tilde{f}_n'\}_{n=1}^\infty$. 
Hence $G$ does not have property $(T_{L^\Phi([0,1],\mathbb{K})})$ for gauge norm. 
Using (\ref{eq:gaugeorlicz}), we can prove for Orlicz norm. 
\end{proof}

\section{Generalized Mazur map}

Delpech proved the H\"older continuity of a generalized Macer map on the unit sphere of real reflexive Orlicz spaces in \cite{MR2157375}. 
In this section, we see the H\"older continuity of a generalized Mazur map around the unit sphere of real or complex reflexive Orlicz spaces.  

Let $\Omega =[0,1]$ or $ \mathbb{N}$, and $\mathbb{K} = \mathbb{R} $ or $\mathbb{C} $, 
and denote by $L^\Phi $ as $L^\Phi (\Omega,\mathbb{K})$.  

\begin{definition}
Let $\Phi, \Psi $ be two N-functions. 
The map 
$$\phi_{\Phi\Psi} : L^{\Phi} \to L^{\Psi} ; 
f \mapsto \phi_{\Phi\Psi}(f):=\Psi^{-1} \circ \Phi(|f|) \sign (f)  $$
is called the {\it generalised Mazur map},  where $\sign (f)(x) := f(x)/|f(x)|$ for $x\in \Omega $ with $f(x) \not=0$. 
\end{definition}
Note that if $\Phi, \Psi  \in \Delta_2^\Omega$, 
then $ \phi_{\Phi\Psi} $ is a bijection between the unit sphere $S_\Phi$ of $L^{\Phi}$ and the unit sphere $S_\Psi$ of $L^{\Psi}$. 

\begin{theorem}\label{theorem:unifhomeo}
Let $\Phi$ and $\Psi$ be N-functions with $\Phi,\Psi \in \Delta_2^\Omega \cap \nabla_2^\Omega$. 
Then the generalized Mazur map $\phi_{\Phi\Psi}:A_\Phi \to L^{\Psi}$ is a $1\wedge \alpha$-H\"older map for some $0<\alpha<\infty$, where $A_\Phi=\{f\in L^{\Phi} \mid \frac{1}{2}\le \|f\|_{(\Phi)} \le \frac{3}{2}\}$.
\end{theorem}

For $0<\alpha \le \beta < \infty$, a non-decreasing continuous function $\varphi:[0,\infty) \to [0,\infty)$ with $\varphi(0)=0$ is said to be {\it in the class $\mathcal{K} (\alpha, \beta )$}, and denoted as $\varphi \in \mathcal{K} (\alpha, \beta )$, if $\frac{\varphi(t)}{t^\alpha}$ is a non-decreasing function of $t > 0$ and $\frac{\varphi(t)}{t^\beta}$ is a non-increasing function  of $t > 0$. 
Hence for $\varphi \in \mathcal{K} (\alpha, \beta )$, $r \ge 1$ and $0<s \le 1$, we have 
\begin{eqnarray} \label{eq:classK}
r^\alpha \varphi(t) \le \varphi (rt), \ \ \ \ 
\varphi (st) \le s^\alpha \varphi(t), \ \ \ \ 
\varphi (rt) \le r^\beta \varphi(t), \ \ \ \ 
s^\beta \varphi(t) \le \varphi (st). 
\end{eqnarray} 
As Remark 2.2 (i) and Proposition 2.3 in \cite{MR2157375}, for $\Phi \in \Delta_2^\Omega \cap \nabla_2^\Omega$, there exist constants $0<p_\Phi \le q_\Phi <\infty$, $D>0$, $ C>0 $ and an N-function $\tilde{\Phi} \in \mathcal{K} (p_\Phi, q_\Phi)$ such that $$D\Phi(t) \le \tilde{\Phi}(t) \le C\Phi(t)$$ for all $t\ge t_0$ if $\Omega=[0,1]$ (for all $0<t\le t_0$ if $\Omega = \mathbb{N}$). 
Then $L^{\tilde{\Phi}} = L^{\Phi} $ as set, and the identity map is isomorphism. 
Hence the map $A_\Phi\ni f \mapsto \frac{\|f\|_{(\Phi)}}{\|f\|_{(\tilde \Phi)}}f \in A_{\tilde \Phi}$ is a bi-Lipschitz homeomorphism. 
Thus we may assume $\Phi \in \mathcal{K} (p_\Phi, q_\Phi )$ and $\Psi \in \mathcal{K} (p_\Psi, q_\Psi)$.

For $\alpha=\frac{p_\Phi}{q_\Psi}$ and $\beta = \frac{q_\Phi}{p_\Psi} $, the non-decreasing continuous function $\varphi = \Psi^{-1} \circ \Phi $ is in the class $\mathcal{K} (\alpha, \beta)$. 
Then for $f \in L^{\Phi}$ with $\|f\|_{(\Phi)} \ge 1$, by the inequalities (\ref{eq:classK}), 
since $\varphi\left(\frac{|f(x)|}{\|f\|_{(\Phi)}}\right) \le \frac{\varphi\left(|f(x)|\right)}{\|f\|^\alpha_{(\Phi)}}$, we have 
\begin{eqnarray*}
1
=\left\|\phi_{\Phi\Psi}\left(\frac{f}{\|f\|_{(\Phi)}}\right)\right\|_{(\Psi)} 
= \left\|\varphi\left(\frac{|f|}{\|f\|_{(\Phi)}}\right)\right\|_{(\Psi)}
\le \left\| \frac{\varphi\left(|f|\right)}{\|f\|^\alpha_{(\Phi)}}\right\|_{(\Psi)}
= \frac{\left\|\phi_{\Phi\Psi}\left(f\right)\right\|_{(\Psi)}}{\|f\|^\alpha_{(\Phi)}}.
\end{eqnarray*} 
Hence $ \|f\|^\alpha_{(\Phi)} \le \left\|\phi_{\Phi\Psi}\left(f\right)\right\|_{(\Psi)}$. 
Similarly, we have $ \left\|\phi_{\Phi\Psi}\left(f\right)\right\|_{(\Psi)}  \le  \|f\|^\beta_{(\Phi)}$. 
For $f \in L^{\Phi} $ with $0<\|f\|_{(\Phi)} \le 1$, 
we have $\|f\|^\beta_{(\Phi)} \le \left\|\phi_{\Phi\Psi}\left(f\right)\right\|_{(\Psi)}  \le  \|f\|^\alpha _{(\Phi)}$. 
That is, for $f \in L^{\Phi}$,  
\begin{equation}\label{eq:Mazurmap}
\min\{\|f\|^\alpha_{(\Phi)}, \|f\|^\beta_{(\Phi)}\} \le \left\|\phi_{\Phi\Psi}\left(f\right)\right\|_{(\Psi)}  \le  \max\{ \|f\|^\alpha _{(\Phi)}, \|f\|^\beta_{(\Phi)}\}
\end{equation}
holds. 

\begin{lemma} \label{lem:Orlnormineq} 
Let $\varphi\in \mathcal{K} (\alpha, \beta )$. 
Then for all $a,b \in \mathbb{C}$ with $a,b\not=0$ we have: 
\begin{itemize}
\item If $\beta \le 1$, then 
$$|\varphi(|a|)\sign (a) - \varphi(|b|)\sign (b)| \le \varphi(|a-b|) + 4\frac{|a-b|}{|a|+|b|}\varphi(|a|+|b|). $$
\item If $\beta \ge 1$, then 
$$|\varphi(|a|)\sign (a) - \varphi(|b|)\sign (b)| \le (2\beta+4)\frac{|a-b|}{|a|+|b|}\varphi(|a|+|b|). $$
\end{itemize}
\end{lemma}

\begin{proof}
We have 
\begin{eqnarray*}
&& |\varphi(|a|)\sign (a) - \varphi(|b|)\sign (b)| \\ 
&\le & |\varphi(|a|)\sign (a) - \varphi(|b|)\sign (a)| + |\varphi(|b|)\sign (a) - \varphi(|b|)\sign (b)| \\
&\le & |\varphi(|a|) - \varphi(|b|)| + |\sign (a) - \sign (b)|\varphi(|a|+|b|). 
\end{eqnarray*}
Hence
\begin{eqnarray*}
&& |\sign (a) - \sign (b)|\varphi(|a|+|b|) \\ 
&\le & \frac{|\sign (a) - \sign (b)|(|a|+|b|)}{|a|+|b|}\varphi(|a|+|b|) \\
&=& \frac{||a|\sign (a) - |a|\sign (b)|+||b|\sign (a) - |b|\sign (b)|}{|a|+|b|}\varphi(|a|+|b|) \\
&\le & \frac{|a - b|+|b - |a|\sign (b)|+||b|\sign (a) - a|+|a - b|}{|a|+|b|}\varphi(|a|+|b|) \\
&\le & \frac{|a - b|+||b| - |a||+||b| - |a||+|a - b|}{|a|+|b|}\varphi(|a|+|b|) \\
&\le & 4\frac{|a - b|}{|a|+|b|}\varphi(|a|+|b|)
\end{eqnarray*}
If $|b| = |a|$, then the lemma was proved.  
We can suppose that $0 < |b| <  |a|$. 
If $\beta \le 1$, then $\frac{\varphi(t)}{t}$ is a non-increasing function of $t$. 
Hence 
\begin{eqnarray*}
%|\varphi(|a|) - \varphi(|b|)| 
%&=& 
\varphi(|a|) - \varphi(|b|) %\\  
&=& \frac{|b|}{|a|} \varphi(|a|) + \frac{|a|-|b|}{|a|} \varphi(|a|) - \varphi(|b|) \\ 
&\le & |b| \frac{\varphi(|b|)}{|b|} + (|a|-|b|)\frac{\varphi(|a|-|b|)}{|a|-|b|}  - \varphi(|b|)\\ 
&= & \varphi(|a|-|b|) \\ 
&\le & \varphi(|a-b|). 
\end{eqnarray*}
On the other hand, if $\beta \ge 1$, then $1-t^\beta \le \beta(1-t)$ for $t\in [0,1]$.  
Using this inequality, since $\varphi$ is in the class $\mathcal{K} (\alpha, \beta )$, 
we have 
\begin{eqnarray*}
%|\varphi(|a|) - \varphi(|b|)| 
%&=& 
\varphi(|a|) - \varphi(|b|) %\\  
&= & \varphi(|a|) - |b|^\beta \frac{\varphi(|b|)}{|b|^\beta}\\  
&\le & \varphi(|a|) - |b|^\beta \frac{\varphi(|a|)}{|a|^\beta}\\  
&= & \left(1- \frac{|b|^\beta }{|a|^\beta} \right) |a|^\alpha \frac{\varphi(|a|)}{|a|^\alpha }\\  
&\le & \beta \left(1- \frac{|b|}{|a|} \right)  \frac{|a|^\alpha}{(|a|+|b|)^\alpha }\varphi(|a|+|b|). 
\end{eqnarray*}
If $\alpha \ge 1$, then $\frac{|a|^\alpha}{(|a|+|b|)^\alpha } \le 1 \le 2\frac{|a|}{|a|+|b|}$. 
If $\alpha \le 1$, then $$\frac{|a|^\alpha}{(|a|+|b|)^\alpha } = \left( \frac{|a|+|b|}{|a|}\right)^{1-\alpha} \frac{|a|}{|a|+|b| } \le 2\frac{|a|}{|a|+|b|}.$$ 
Hence 
\begin{eqnarray*}
\beta \left(1- \frac{|b|}{|a|} \right)  \frac{|a|^\alpha}{(|a|+|b|)^\alpha }\varphi(|a|+|b|)
&\le & 2\beta \left(1- \frac{|b|}{|a|} \right)  \frac{|a|}{|a|+|b|}\varphi(|a|+|b|) \\ 
&=& 2\beta  \frac{|a|-|b|}{|a|+|b|}\varphi(|a|+|b|) \\ 
&\le & 2\beta  \frac{|a-b|}{|a|+|b|}\varphi(|a|+|b|) . 
\end{eqnarray*}
This proves the lemma. 
\end{proof}
Using above lemma, we can prove Theorem \ref{theorem:unifhomeo} as in \cite{MR2157375}.  

\begin{proof}[Proof of Theorem \ref{theorem:unifhomeo}]
We may assume $\varphi = \Psi^{-1} \circ \Phi \in \mathcal{K} (\alpha, \beta )$ for some $0<\alpha \le \beta <\infty$. 
Fix $f,h \in A_\Phi$ with $f \not= h$.  
Let 
$$\Delta_{\Phi\Psi}(x) = |\phi_{\Phi\Psi}(f)(x)-\phi_{\Phi\Psi}(h)(x)| = |\varphi(|f(x)|)\sign(f(x))-\varphi(|h(x)|)\sign(h(x))|, $$
$v(x) = |f(x)-h(x)|$ and $w(x) =|f(x)| + |h(x)|$ for $x \in \Omega $.  
Our aim is to estimate 
$$\| \phi_{\Phi\Psi}(f)-\phi_{\Phi\Psi}(h) \|_{(\Psi)}  = \|\Delta_{\Phi\Psi}\|_{(\Psi)} =  \inf \left\{ b> 0 \mid \int_\Omega \Psi \left( \frac{\Delta_{\Phi\Psi}}{b}\right) d\mu(x) \le 1\right\} $$
using $\| f-h\|_{(\Phi)} = \|v\|_{(\Phi)}$. 
We show the estimate for the three cases $\alpha \le \beta \le 1$, $\alpha \le 1 \le \beta$, and $ 1\le \alpha \le \beta$. 

\ 

\noindent
{\bf Case 1}: $\alpha \le \beta \le 1$.

Let $b=\frac{\|v\|_{(\Phi)}^\alpha}{8}$, then $0<b<1$. 
Since $\beta \le 1$, using Lemma \ref{lem:Orlnormineq} we have 
$$ \frac{\Delta_{\Phi\Psi}(x) }{b}
\le \frac{1}{b} \varphi(v(x)) + \frac{4}{b}\frac{v(x)}{w(x)}\varphi(w(x)) 
\le \frac{4}{b} \varphi(v(x)) + \frac{4}{b}\frac{v(x)}{w(x)}\varphi(w(x)).  $$
If $\frac{4}{b}\le 1$ and $\frac{4}{b}\frac{v(x)}{w(x)} \le 1$, then since $\frac{1}{\beta}-1 \ge 0$, $\frac{v(x)}{w(x)}\le 1$, and $b<1$, 
using inequalities (\ref{eq:classK}) we have 
\begin{eqnarray*}
\frac{4}{b} \varphi(v(x)) + \frac{4}{b}\frac{v(x)}{w(x)}\varphi(w(x)) 
&\le & \varphi \left(\left(\frac{4}{b}\right)^\frac{1}{\beta} v(x)\right) + \varphi\left(\left(\frac{4}{b}\frac{v(x)}{w(x)}\right)^\frac{1}{\beta} w(x)\right) \\
%&=& \varphi \left(\left(\frac{4}{b}\right)^\frac{1}{\beta} v(x)\right) + \varphi\left(\left(\frac{4}{b}\right)^\frac{1}{\beta} \left(\frac{v(x)}{w(x)}\right)^{\frac{1}{\beta}-1} v(x)\right) \\
&\le & 2 \varphi \left(\left(\frac{4}{b}\right)^\frac{1}{\beta} v(x)\right)  \\
&\le & \varphi \left(2^\frac{1}{\alpha}\frac{4^\frac{1}{\beta}}{b^\frac{1}{\beta}} v(x)\right) \\ 
&\le & \varphi \left(\frac{8^\frac{1}{\alpha}}{b^\frac{1}{\alpha}} v(x)\right) .  
\end{eqnarray*}
Similarly, for other cases ($\frac{4}{b} > 1$ or $\frac{4}{b}\frac{v(x)}{w(x)} >\ 1$) we have 
\begin{eqnarray*}
\frac{4}{b} \varphi(v(x)) + \frac{4}{b}\frac{v(x)}{w(x)}\varphi(w(x)) 
\le \varphi \left(\left(\frac{8}{b}\right)^\frac{1}{\alpha} v(x)\right) .  
\end{eqnarray*}
Hence
\begin{eqnarray*}
\int_\Omega \Psi \left( \frac{\Delta_{\Phi\Psi}(x)}{b}\right) d\mu(x) 
&\le& \int_\Omega \Psi \left( \varphi \left(\left(\frac{8}{b}\right)^\frac{1}{\alpha} v(x)\right) \right) d\mu(x)  \\
&=& \int_\Omega \Phi \left( \frac{v(x)}{\|v\|_{(\Phi)}} \right) d\mu(x) \\ 
&=& 1. 
\end{eqnarray*}
This means $\|\Delta_{\Phi\Psi}\|_{(\Psi)} \le \frac{\|v\|_{(\Phi)}^\alpha}{8}$ . 
Hence $\phi_{\Phi\Psi}$ is $\alpha$-H${\rm \ddot{o}}$lder on $A_\Phi$.

\ 

\noindent
{\bf Case 2}: $1\le \alpha \le \beta $.

Let $b= 3^{\beta} (2\beta+4)\|v\|_{(\Phi)}$. 
Since $\beta \ge 1$, using Lemma \ref{lem:Orlnormineq} and inequalities (\ref{eq:classK}) we have 
$$ \frac{\Delta_{\Phi\Psi}(x) }{b}
\le \frac{2\beta+4}{b}\frac{v(x)}{w(x)}\varphi(w(x)) 
\le \frac{3^\beta (2\beta+4)}{b}\frac{v(x)}{w(x)}\varphi\left(\frac{w(x)}{3}\right) .  $$
Let $\Omega_1:=\{ x\in \Omega \mid \frac{3^\beta (2\beta+4)}{b}\frac{v(x)}{w(x)} \le 1\} $ and $\Omega_2:=\{ x\in \Omega \mid \frac{3^\beta (2\beta+4)}{b}\frac{v(x)}{w(x)} > 1\} $. 
If $x \in \Omega_1 $, by the convexity of $\Psi$ 
\begin{eqnarray*}
\Psi\left( \frac{\Delta_{\Phi\Psi}(x)}{b}\right)
&\le & \Psi\left( \frac{3^\beta (2\beta+4)}{b}\frac{v(x)}{w(x)}\varphi\left(\frac{w(x)}{3}\right) \right) \\ 
&\le & \frac{3^\beta (2\beta+4)}{b}\frac{v(x)}{w(x)} \Psi\left( \varphi\left(\frac{w(x)}{3}\right) \right) \\ 
&=& \frac{3^\beta (2\beta+4)}{b}\frac{v(x)}{w(x)} \Phi\left(\frac{w(x)}{3} \right) . 
\end{eqnarray*}
Since 
\begin{eqnarray*}
\left\| \frac{w}{3} \right\|_{(\Phi)} 
= \left\| \frac{|f|+|h|}{3} \right\|_{(\Phi)}
\le  \frac{\left\| f\right\|_{(\Phi)} + \left\| h \right\|_{(\Phi)}}{3} 
\le  1, 
\end{eqnarray*}
using the inequality $\Phi^*\left( \frac{\Phi(t)}{t}\right) \le \Phi(t)$ for $t>0$ (see \cite{MR0126722}, p.13), we have 
\begin{eqnarray*}
\int_{\Omega} \Phi^*\left( \frac{\Phi\left( \frac{w(x)}{3} \right)}{ \frac{w(x)}{3} }\right) d\mu(x)
\le \int_{\Omega} \Phi\left( \frac{w(x)}{3} \right) d\mu(x) \le 1. 
\end{eqnarray*}
Thus 
\begin{eqnarray*}
\left\| \frac{\Phi\left( \frac{w(x)}{3} \right)}{ \frac{w(x)}{3} }\right\|_{(\Phi^*)} 
\le 1. 
\end{eqnarray*}
Using non-normalized H\"older inequality  (see Proposition 1 at 3.3 in \cite{MR1113700}), we have 
\begin{eqnarray*} 
\int_{\Omega_1} \Psi \left( \frac{\Delta_{\Phi\Psi}(x)}{b}\right) d\mu(x)  
&\le& \int_{\Omega_1} \frac{3^\beta (2\beta+4)}{b}\frac{v(x)}{w(x)} \Phi\left(\frac{w(x)}{3} \right)  d\mu(x)  \\ 
%&\le& \int_\Omega \frac{2^\beta (2\beta+4)}{b}\frac{v(x)}{w(x)} \Phi\left(\frac{w(x)}{2} \right)  d\mu(x)  \\ 
&\le& \frac{3^{\beta-1} (2\beta+4)}{b} 2\|v\|_{(\Phi)} \left\| \frac{\Phi\left( \frac{w(x)}{3} \right)}{ \frac{w(x)}{3} }\right\|_{(\Phi^*)}   \\ 
%&\le& \frac{2^{\beta-1} (2\beta+4)\|f-h\|_{(\Phi)} }{b}  \\ 
&\le& \frac{2}{3}. 
\end{eqnarray*}
On the other hand, since $\frac{\varphi(t)}{t}$ is a non-decreasing function, for $x \in \Omega_2$ we have 
\begin{eqnarray*}
\frac{3^\beta (2\beta+4)}{b}\frac{v(x)}{w(x)}\varphi\left(\frac{w(x)}{3}\right)
\le  \varphi\left(\frac{3^\beta (2\beta+4)}{b}\frac{v(x)}{w(x)}\frac{w(x)}{3}\right) 
%= \varphi\left(\frac{3^{\beta-1} (2\beta+4)v(x)}{b}\right) 
= \varphi\left(\frac{v(x)}{3\|v\|_{(\Phi)}}\right) . 
\end{eqnarray*}
Hence by the convexity of $\Phi$
\begin{eqnarray*}
\int_{\Omega_2} \Psi \left( \frac{\Delta_{\Phi\Psi}(x)}{b}\right) d\mu(x)  
%\le \int_\Omega \Psi \left( \varphi\left(\frac{3^{\beta-1} (2\beta+4)v(x)}{b}\right) \right) d\mu(x)  
\le \int_\Omega \Phi \left(\frac{v(x)}{3\|v\|_{(\Phi)}}\right) d\mu(x) 
%\le  \int_\Omega \frac{1}{3}\Phi \left(\frac{v(x)}{\|v\|_{(\Phi)}}\right) d\mu(x) 
\le  \frac{1}{3}. 
\end{eqnarray*}
Summarizing these inequalities, we have 
\begin{eqnarray*}
\int_\Omega \Psi \left( \frac{\Delta_{\Phi\Psi}(x)}{b}\right) d\mu(x)  \le 1. 
\end{eqnarray*}
This means 
$\|\Delta_{\Phi\Psi}\|_{(\Psi)} \le b = 3^{\beta} (2\beta+4)\|v\|_{(\Phi)}$. 
Hence $\phi_{\Phi\Psi}$ is Lipschitz on $A_\Phi$. 

\

\noindent
{\bf Case 3}: $\alpha \le 1 \le \beta$.

Let $b= 3^{\beta-\alpha+1} (2\beta+4)\|v\|_{(\Phi)}^\alpha $. 
Since $\beta \ge 1$, as above 
\begin{eqnarray*} 
\frac{1}{b}\Delta_{\Phi\Psi}(x) 
%\le \frac{2\beta+4}{b}\frac{v(x)}{w(x)}\varphi(w(x)) 
\le \frac{3^\beta (2\beta+4)}{b}\frac{v(x)}{w(x)}\varphi\left(\frac{w(x)}{3}\right) .  
\end{eqnarray*}
Let $\Omega_1:=\{ x\in \Omega \mid \frac{3^\beta (2\beta+4)}{b}\frac{v(x)}{w(x)} \le 1\} $ 
and $\Omega_2:=\{ x\in \Omega \mid \frac{3^\beta (2\beta+4)}{b}\frac{v(x)}{w(x)} > 1\} $. 
Since $\frac{\|v\|_{(\Phi)}}{3}\le \left(\frac{\|v\|_{(\Phi)}}{3}\right)^\alpha$, for $x \in \Omega_1$, 
as above we have 
\begin{eqnarray*}
\int_{\Omega_1} \Psi \left( \frac{\Delta_{\Phi\Psi}(x)}{b}\right) d\mu(x)  
&\le & \frac{3^{\beta-1} (2\beta+4)2\|v\|_{(\Phi)} }{b} \\  
&\le & \frac{3^{\beta-1} (2\beta+4)3^{1-\alpha }2\|v\|_{(\Phi)}^\alpha }{b}  \\
&\le & \frac{2}{3}. 
\end{eqnarray*}
On the other hand, for $x \in \Omega_2$, using inequalities (\ref{eq:classK}) we have 
\begin{eqnarray*}
\frac{3^\beta (2\beta+4)}{b}\frac{v(x)}{w(x)}\varphi\left(\frac{w(x)}{3}\right)
&\le & \varphi\left( \left( \frac{3^\beta (2\beta+4)}{b}\frac{v(x)}{w(x)} \right)^\frac{1}{\alpha} \frac{w(x)}{3}\right) \\ 
&= & \varphi\left( \frac{3^{\frac{\beta}{\alpha}-1} (2\beta+4)^\frac{1}{\alpha}}{b^\frac{1}{\alpha}} \left(\frac{v(x)}{w(x)} \right)^{\frac{1}{\alpha}-1} v(x)\right) \\ 
&\le & \varphi\left( \frac{3^{\frac{\beta}{\alpha}-1} (2\beta+4)^\frac{1}{\alpha}}{b^\frac{1}{\alpha}}  v(x)\right) \\ 
&\le & \varphi\left( \frac{v(x)}{3^\frac{1}{\alpha}\|v\|_{(\Phi)}}\right) . 
\end{eqnarray*}
Hence by the convexity of $\Phi$
\begin{eqnarray*}
\int_{\Omega_2} \Psi \left( \frac{\Delta_{\Phi\Psi}(x)}{b}\right) d\mu(x)  
%\le \int_\Omega \Psi \left( \varphi\left( \frac{v(x)}{3^\frac{1}{\alpha}\|v\|_{(\Phi)}}\right)\right) d\mu(x) 
\le \int_\Omega \Phi \left(\frac{v(x)}{3^\frac{1}{\alpha}\|v\|_{(\Phi)}}\right) d\mu(x) 
%&\le & \int_\Omega \frac{1}{2}\Phi \left(\frac{v(x)}{\|v\|_{(\Phi)}}\right) d\mu(x) \\ 
\le  \frac{1}{3^\frac{1}{\alpha}} \le \frac{1}{3}. 
\end{eqnarray*}
Summarizing these inequalities, we have 
\begin{eqnarray*}
\int_\Omega \Psi \left( \frac{\Delta_{\Phi\Psi}(x)}{b}\right) d\mu(x) \le 1. 
\end{eqnarray*}
This means $\|\Delta_{\Phi\Psi}\|_{(\Psi)} \le b =3^{\beta-\alpha+1} (2\beta+4)\|f-h\|_{(\Phi)}^\alpha$. 	
Hence $\phi_{\Phi\Psi}$ is $\alpha$-H${\rm \ddot{o}}$lder on $A_\Phi$. 
This completes the proof. 
\end{proof}

\section{Proof of Theorem \ref{thm:TtoTPhi}}

Let $G$ be a locally compact second countable group, and $\Phi$ an N-function such that $\Phi \in \Delta_2^\Omega \cap \nabla_2^\Omega$, and  
\begin{center}
$(1)$ $\Omega=[0,1]$ and $\mathbb{K}=\mathbb{R}$,  \ \ or \ \ \
$(2)$ $\Omega=[0,1]$ and $\mathbb{K}=\mathbb{C}$,  \ \ or \ \ \ 
$(3)$ $\Omega=\mathbb{N}$ and $\mathbb{K}=\mathbb{C}$. %and the standard unit vectors form an 1-symmetric basis in $L^\Phi(\Omega,\mathbb{K})$,  
\end{center}
By Proposition 4 at 3.2 in \cite{MR1113700} the Orlicz space $L^\Phi([0,1], \mathbb{K})$ with the gauge norm $\|\ \|_{(\Phi)}$ %has Fatou's property.  
is a rearrangement-invariant function space. 
Hence by Theorem 10 in \cite{MR0158259}, Theorem 1.1 in \cite{MR1295579}, by Theorem 1 in \cite{MR771996} (since $\Omega=\mathbb{N}$ has counting measure),  
%by Theorem 5.2.14 in \cite{MR1957004},   
for a surjective linear isometry $U$ on $L^{\Phi}(\Omega, \mathbb{K})$, 
there exist a Borel function $h:\Omega \to \mathbb{R}$, % which is $|h(x)|=1$ a.e., 
and an invertible Borel map $T:\Omega \to \Omega$  
such that 
\begin{enumerate}
\item[{\rm (i)}] for any Borel set $A\subset \Omega $, $\mu (T^{-1}A)=0$ if and only if $\mu (A)=0$, and 
\item[{\rm (ii)}] for all $f\in L^{\Phi}(\Omega)$ 
\begin{equation}\label{eq:U=hT_1Orl} 
Uf(x) = h(x)f(T(x))\ \ \    \text{ a.e. } x\in \Omega.  
\end{equation}
\end{enumerate}
In particular, if $\Omega =\mathbb{N}$, then $|h(x)|=1$ for all $x\in \Omega$. 
When $\Omega=[0,1]$, since $\mu \circ T$ is a measure which is absolutely continuous with respect to $\mu $, it has the Radon-Nykodym derivative $r:[0,1] \to [0,\infty)$, which satisfies 
\begin{eqnarray*}
\mu (TA) =  \int_A r(x) d\mu, \ \ \ \ \ 
 \int_A f(x) d\mu(x) = \int_{T^{-1}A} f(Tx) r(x)d\mu(x)  %;\ \ \ \  \int_{TA} f d\mu = \int_A T_1^{-1}(f[g^{-1}]) d\mu 
\end{eqnarray*}
for any Borel set $A\subset [0,1]$ and any $f\in L^{\Phi}([0,1],\mathbb{K} )$. 
Under this situation, by the same proof of Theorem 5.4.10 (106) in \cite{MR1957004}, the equation (\ref{eq:U=hT_1Orl}) implies  
\begin{equation}\label{eq:trans-theta} 
\Phi (|h(x)|\alpha ) = r(x) \Phi (\alpha)
\end{equation}
for almost all $x\in [0,1]$ and all $\alpha \ge 0$.

%\begin{theorem}%[Theorem 5.4.16 in \cite{MR1957004}]
%Under the situation in Theorem \ref{thm:UinLtheta}, the groups of linear isometries of $(L^\Phi,N_\Phi)$ and $(L^\Phi,\|\ \|_\Phi)$ coincides. 
%\end{theorem}

\begin{lemma}
The conjugation 
$$U \mapsto \phi_{\Phi, t^2} \circ U \circ \phi_{t^2, \Phi}$$ 
is a homomorphism from $O(L^{\Phi}(\Omega, \mathbb{K}))$ to $O(L^2(\Omega, \mathbb{K}))$. 
\end{lemma}

\begin{proof}
For a simple function $f\in L^2(\Omega, \mathbb{K})$, using the equations (\ref{eq:U=hT_1Orl}) and (\ref{eq:trans-theta}) we have 
\begin{eqnarray*} 
&& \phi_{\Phi, t^2} \circ U \circ \phi_{t^2, \Phi} (f)(x) \\
&=&  \phi_{\Phi, t^2} \circ U \left(\Phi^{-1}(|f(x)|^2) \sign (f(x))\right)  \\
&=& \phi_{\Phi, t^2} \left( h(x) \Phi^{-1}(|f(Tx)|^2) \sign (f(Tx)) \right) \\ 
&=& \left( \Phi \left( |h(x) \Phi^{-1}(|f(Tx)|^2) \sign (f(Tx)) |\right)\right)^{\frac{1}{2}} \sign \left( h(x)  \sign (f(Tx)) \right)\\ 
&=& \left( r(x)\Phi \left( \Phi^{-1}(|f(Tx)|^2) \right)\right)^{\frac{1}{2}} \sign ( h(x) ) \sign (f(Tx))\\ 
&=& r(x)^{\frac{1}{2}}\sign ( h(x) )  f(Tx) 
\end{eqnarray*}
Hence the map $U\mapsto  \phi_{\Phi, t^2} \circ U \circ \phi_{t^2, \Phi}$ is linear. 
Furthermore, since $\Omega=T^{-1}\Omega$, we have 
\begin{eqnarray*}
\| \phi_{\Phi, t^2} \circ U \circ \phi_{t^2, \Phi} (f)\|^2_{L^2} 
%&=& \int_{\Omega} |\phi_{\Phi, t^2} \circ U \circ \phi_{t^2, \Phi} (f)(x)|^2 d\mu(x) \\
&=& \int_{\Omega} |r(x)^{\frac{1}{2}}\sign ( h(x) ) f(Tx)  |^2 d\mu(x)  \\
&=& \int_{T^{-1}\Omega} |f(Tx) |^2 r(x) d\mu(x)  \\
&=& \int_{\Omega} |f(x) |^2 d\mu(x)  \\ 
&=& \|f \|^2_{L^2}. 
\end{eqnarray*}
If we define $(\phi_{\Phi, t^2} \circ U \circ \phi_{t^2, \Phi})^{-1} := \phi_{\Phi, t^2} \circ U^{-1} \circ \phi_{t^2, \Phi}$, 
then 
\begin{eqnarray*}
(\phi_{\Phi, t^2} \circ U \circ \phi_{t^2, \Phi})^{-1}\phi_{\Phi, t^2} \circ U \circ \phi_{t^2, \Phi} 
&=& \phi_{\Phi, t^2} \circ U^{-1} \circ \phi_{t^2, \Phi} \circ \phi_{\Phi, t^2} \circ U \circ \phi_{t^2, \Phi}  \\
&=& \operatorname{id}. 
\end{eqnarray*} 
Hence $\phi_{\Phi, t^2} \circ U \circ \phi_{t^2, \Phi} $ is an invertible linear isometry on the subspace $\{$simple functions$\}\subset L^2(\Omega, \mathbb{K})$. 
By extending the map $U\mapsto  \phi_{\Phi, t^2} \circ U \circ \phi_{t^2, \Phi}$ to the linear isometry on  $L^2(\Omega, \mathbb{K})$, we have the homomorphism from $O(L^\Phi(\Omega, \mathbb{K}))$ to $O(L^2(\Omega, \mathbb{K}))$. 
\end{proof}

\begin{theoremTtoTPhi}
Let $G$ be a locally compact second countable group, and $\Phi$ an N-function such that $\Phi \in \Delta_2^\Omega \cap \nabla_2^\Omega$, and  
\begin{center}
$(1)$ $\Omega=[0,1]$ and $\mathbb{K}=\mathbb{R}$,  \ \ or \ \ \
$(2)$ $\Omega=[0,1]$ and $\mathbb{K}=\mathbb{C}$,  \ \ or \ \ \ 
$(3)$ $\Omega=\mathbb{N}$ and $\mathbb{K}=\mathbb{C}$. %and the standard unit vectors form an 1-symmetric basis in $L^\Phi(\Omega,\mathbb{K})$,  
\end{center}
If $G$ has Kazhdan's property $(T)$, then it has property $(T_{L^\Phi(\Omega,\mathbb{K})})$ with respect to gauge norm. 
\end{theoremTtoTPhi}

\begin{proof}
%We follow the arguments in \cite{MR2316269}. 
Assume $G$ does not have property $(T_{L^\Phi(\Omega, \mathbb{K})})$. 
Write $B=L^\Phi(\Omega, \mathbb{K})$ and $H=L^2(\Omega, \mathbb{K})$ $(H=L^2(\Omega \times \{1,\sqrt{-1}\}, \mathbb{R})$ if $\mathbb{K}=\mathbb{C}$). 
Then there is a linear isometric $G$-representation $\rho:G\to O(B)$ so that the quotient representation $\tilde{\rho}: G\to O(B/B^{\rho(G)})$ almost has invariant vectors, 
i.e. for any compact subset $K \subset G$ and $n\in \mathbb{N}$, there exist unit vectors $\tilde{f}_n = f_n + B^{\rho(G)} \in \tilde{B} := B/B^{\rho(G)}$ so that 
$$\max_{g \in K} \|\tilde{\rho}(g )\tilde{f}_n-\tilde{f}_n\|_{\tilde{B}} 
= \max_{g \in K} \inf_{h\in B^{\rho(G)}} \|\rho (g )f_n-f_n -h\|_B <\frac{1}{n^2}  . $$ 
Hence for each $g \in K$ and $n\in \mathbb{N}$, there is  $h_{g,n} \in B^{\rho(G)}$ such that 
$$ \|(\rho (g )f_n-f_n) -h_{g,n}\|_B <\frac{1}{n^2}  . $$ 
Thus for $i\in \mathbb{N}$ we have 
$$ \|(\rho(g^i) f_n-\rho(g^{i-1})f_n) -h_{g,n} \|_B = \|\rho(g^{i-1}) (\rho (g )f_n-f_n -h_{g,n}) \|_B <\frac{1}{n^2}  . $$ 
Hence 
$$ \|(\rho(g^n) f_n-f_n) -nh_{g,n} \|_B \le \sum_{i=1}^n \|(\rho(g^i) f_n-\rho(g^{i-1})f_n) -h_{g,n} \|_B <\frac{1}{n}  . $$ 
On the other hand, since $\|\tilde{f}_n\|_{\tilde{B}}=1$, there is $h_n \in B^{\rho(G)}$ such that $$1\le \|f_n - h_n\|_B\le 1+ \frac{1}{n^2}.$$ 
Hence 
$$\|\rho(g^n) f_n-f_n\|_B \le \|\rho(g^n) f_n- \rho(g^n) h_n\|_B + \|f_n-h_n\|_B \le 2 + \frac{2}{n^2}.$$
Since  
\begin{eqnarray*}
n \|h_{g,n} \|_B 
&=& \|n h_{g,n} \|_B  \\ 
&\le & \|(\rho(g^n) f_n-f_n) -nh_{g,n} \|_B + \|\rho(g^n) f_n-f_n\|_B \\ 
&\le & 2 + \frac{1}{n} + \frac{2}{n^2}, 
\end{eqnarray*}
we get 
$$ \|\rho (g )f_n-f_n\|_B \le \|(\rho (g )f_n-f_n) -h_{g,n}\|_B + \|h_{g,n}\|_B \le \frac{2}{n} + \frac{2}{n^2} + \frac{2}{n^3}  \le \frac{6}{n}. $$
Thus the sequences $\{f'_n= \frac{f_n}{\|f_n - h_n\|_B}\}_{n\in \mathbb{N}} \subset B$, $\{h'_n = \frac{h_n}{\|f_n - h_n\|_B}\}_{n\in \mathbb{N}} \subset B^{\rho(G)}$ satisfy $\|f'_n - h'_n\|_B = 1$ and 
\begin{eqnarray*}
\max_{g\in K} \|\rho (g )(f'_n-h'_n)-(f'_n-h'_n)\|_B 
&=& \frac{\max_{g\in K} \|\rho (g )(f_n-h_n)-(f_n-h_n)\|_B}{\|f_n - h_n\|_B} \\ 
&\le & \frac{6}{n}. 
\end{eqnarray*}
Let us then define $\pi: G\to O(\mathcal{H})$ by $\pi(g ) = \phi_{\Phi, t^2} \circ \rho(g ) \circ \phi_{t^2, \Phi} $. 
Then $\phi_{\Phi,t^2} $ maps $B^{\rho(G)}$ onto $\mathcal{H}^{\pi(G)}$.  
Let $v_n := \phi_{\Phi,t^2}(f'_n-h'_n)$. 
Then $\|v_n\|_H=1$ and by Theorem \ref{theorem:unifhomeo} there are $0<\alpha <\infty$ and a constant $C>0$ such that 
\begin{eqnarray*}
\max_{g\in K} \|\pi (g )v_n-v_n\|_H  
&=& \max_{g\in K} \|\phi_{\Phi,t^2}(\rho (g )(f'_n-h'_n))-\phi_{\Phi,t^2}(f'_n-h'_n)\|_H \\ 
&\le & C \left(\frac{6}{n}\right)^{1\wedge \alpha}. 
\end{eqnarray*}
From the inequalities (\ref{eq:Mazurmap}), there is $\delta > 0$ such that if $1-\delta \le \|u\|_H \le 1+\delta $, then $\frac{1}{2} \le \|\phi_{t^2,\Phi}(u)\|_H \le \frac{3}{2}$. 
For $n\in \mathbb{N}$ and $u\in H^{\pi(G)}$,  
if $  \|u\|_H < 1-\delta $, then 
$$\| v_n -u\|_H \ge \| v_n \|_H - \|u\|_H > \delta ,  $$
if $  \|u\|_H > 1+\delta $, then 
$$\| v_n -u\|_H \ge \|u\|_H - \| v_n \|_H > \delta ,  $$
if $1-\delta \le \|u\|_H \le 1+\delta $, then 
\begin{eqnarray*}
C \| v_n -u\|_H ^{1\wedge \alpha }
&\ge & \|\phi_{t^2,\Phi}(v_n) -\phi_{t^2,\Phi}(u)\|_B \\
&=& \|f'_n -h'_n -\phi_{t^2,\Phi}(u)\|_B \\ 
&\ge & \inf_{h\in B^{\rho(G)}}\|f'_n -h\|_B \\ 
&=& \inf_{h\in B^{\rho(G)}} \left\| \frac{f_n}{\|f_n - h_n\|_B} -h \right\|_B \\ 
&=& \frac{\inf_{h\in B^{\rho(G)}}\|f_n-h\|_B }{\|f_n - h_n\|_B} \\ 
&\ge & \frac{1}{1+\frac{1}{n^2}} 
\ge \frac{1}{2}\\ 
\end{eqnarray*}
for some $C> 0$ and $0< \alpha < \infty$. 
That is, there is a constant $\delta' >0$ such that for all $n \in \mathbb{N}$ 
$$ \inf_{u\in H^{\pi(G)}}\|v_n-u \|_H \ge \delta' .$$
Let $w_n$ denote the projection of $v_n$ to $H' = (H^{\pi(G)})^{\perp}$. 
Then $\|w_n\|_{H}  \ge \delta' >0$ for all $n$ and 
$$\max_{g \in K}\left\|\pi(g ) \frac{w_n}{\|w_n\|_H} - \frac{w_n}{\|w_n\|_H} \right\|_H  
\le \max_{g \in K}\frac{1}{\delta'}\|\pi(g )v_n -v_n\|_H 
\to 0 \ \ \ \text{ as } n \to \infty .$$ 
Thus the restriction $\pi'$ of $\pi $ to $H'$ does not $G$-invariant vectors, but almost does. 
Hence $G$ does not have Kazhdan's property $(T)$. 
\end{proof}

\section{Proof of Theorem \ref{thm:TtoFPhi}}

\begin{theoremTtoFPhi} 
If a locally compact second countable topological group $G$ has property $(T)$, then there exists a constant $\epsilon (G)>0$ such that $G$ has property $(F_{B})$ for every real $($or complex$)$ Banach spaces $B$ with $\delta_B(t) \ge \delta_{L^{2+\epsilon(G)}}(t)$ for all $0<t<2$ $($or  $\rho_B(t) \le \rho_{L^{2+\epsilon(G)}}(t)$ for all $t>0)$. 
\end{theoremTtoFPhi}

\begin{proof}
Since $G$ is a locally compact second countable topological group with Kazhdan's property $(T)$, it is compactly generated. 
Fix a compact generating subset $K$ with non-empty interior of $G$. 

\begin{lemma}
There exist a constant $\epsilon(G) >0$  and $C<\infty$ such that for any real $($or complex$)$ Banach space $(B,\|\ \|)$ with $\delta_B(t) \ge \delta_{L^{2+\epsilon(G)}}(t)$ for all $0<t<2$ $($or $\rho_B(t) \le \rho_{L^{2+\epsilon(G)}}(t)$ for all $t>0)$, any affine isometric action $\alpha$ of $G$ on $B$, and any point $x\in B$ with $\max_{g\in K}\|\alpha(g,x) -x\|>0$, there exists a point $y\in B$ with 
$$\|x-y\| \le C \max_{g\in K}\|\alpha(g,x) -x\|, \ \ \ \max_{g\in K}\|\alpha(g,y) -y\| \le \frac{\max_{g\in K}\|\alpha(g,x) -x\|}{2}.$$
\end{lemma}

\begin{proof}
By contradiction, we assume for any $n\in \mathbb{N}$ there exist 
Banach spaces $(B_n,\|\ \|_n)$ with $\delta_{B_n}(t) \ge \delta_{L^{2+\frac{1}{n}}}(t)$ for all $0<t<2$, 
affine isometric $G$-actions $\alpha_n$ on $B_n$, and 
points $x_n \in B_n$ such that, after a rescaling to achieve 
$\max_{g\in K}\|\alpha_n(g,x_n) -x_n\|_n =1$, 
$$ \max_{g\in K}\|\alpha_n(g,y) -y\|_n > \frac{1}{2} $$
for all $y \in B_n $ with $\|y-x_n\|_n \le n$.  

Let $\omega $ be a non-principal ultrafilter. 
Set $(B_{\omega }, \|\ \|_{\omega }) $ be the ultraproduct of the spaces $(B_n,\|\ \|_n)$ with the marked points $x_n$.  
For $0< \epsilon \le 2$, we take $u,v\in B_{\omega }$ with $\|u\|_{\omega } = \|v\|_{\omega } =1$ and $\|u-v\|_{\omega } \ge \epsilon $. 
Let $(u_n)$ be the representatives of $u$, $(v_n)$ of $v$, and $0<\eta <\epsilon $.  
Then by the uniform convexity 
\begin{eqnarray*} 
&& \{n\in \mathbb{N} : |\ \|u_n\|_n -1|< \eta , \ |\ \|v_n\|_n -1| < \eta ,\ |\|u-v\|_{\omega } - \|u_n-v_n\|_n|<\eta \} \\ %\in \omega . 
&\subset & \{n\in \mathbb{N} : |\ \|u_n\|_n -1|< \eta , \ |\ \|v_n\|_n -1| < \eta ,\ \|u_n-v_n\|_n > \epsilon -\eta \} \\ %\in \omega . 
&\subset& \left\{ n\in \mathbb{N} : 
\frac{\|u_n\|_n }{1+\eta }< 1 ,\ 
\frac{\|v_n\|_n}{1+\eta} < 1 ,\ 
\frac{\|u_n-v_n\|_n}{1+\eta} > 
\frac{\epsilon -\eta }{1+\eta } \right\} \\
&\subset & \left\{n\in \mathbb{N} : \frac{\|u_n + v_n\|_n }{2(1+\eta )} \le 
1-\delta_{B_n} 
\left(\frac{\epsilon -\eta }{1+\eta }\right) \right\} . 
\end{eqnarray*} 
Since the set at the top line is in $\omega$, the set at the bottom line is also in $\omega$.  
Since $\delta_{L^{2+\frac{1}{m}}} \le \delta_{L^{2+\frac{1}{n}}} $ for $m \le n$, for fixed $m\in \mathbb{N}$ we have 
\begin{eqnarray*} 
\left\{n\in \mathbb{N} : \frac{\|u_n + v_n\|_n }{2(1+\eta )} \le 
1-\delta_{L^{2+\frac{1}{m}}}  
\left(\frac{\epsilon -\eta }{1+\eta }\right) \right\} \cap \{n\in \mathbb{N} : n\ge m \}\in \omega . 
\end{eqnarray*} 
Since $\eta $ is arbitrary, we have 
\begin{eqnarray*}
\frac{\|u + v \|_{\omega } }{2} 
\le \inf_{0<\eta <\epsilon  } (1+\eta )\left(1-\delta_{L^{2+\frac{1}{m}}} \left(\frac{\epsilon -\eta }{1+\eta }\right) \right) 
= 1-\delta_{L^{2+\frac{1}{m}}} (\epsilon ), 
\end{eqnarray*}
that is, 
$\delta _{B_{\omega }}(\epsilon )\ge \delta_{L^{2+\frac{1}{m}}} (\epsilon ) $ for $0< \epsilon \le 2$. 
Since $m$ is also arbitrary, we have $\delta _{B_{\omega }}(\epsilon )\ge \delta_{L^2} (\epsilon ) $ for $0< \epsilon \le 2$. 
This means $(B_\omega,\|\ \|_\omega)$ is a Hilbert space (see p.410 in \cite{MR1727673} and Theorem 7.2. in \cite{MR0022312}). 
Since $G$ is generated by $K$ and $\max_{g\in K}\|\alpha(g,x_n)-x_n\|_n=1$, we obtain an affine isometric $G$-action $\alpha_\omega$ on the Hilbert space $B_\omega$ from $\alpha_n$. 
Then, by the assumption, for any $y\in B_\omega$, we have 
$$ \max_{g\in K}\|\alpha_\omega(g,y) -y\|_\omega \ge  \frac{1}{2} , $$
that is, $\alpha_\omega $ has no fixed point. 
This means $G$ does not have property $(FH)$, hence contradicting the property $(T)$ of $G$. 
For $B$ with $\rho_B(t) \le \rho_{L^{2+\epsilon(G)}}(t)$ for all $t>0$, we can prove in a similar manner. 
\end{proof}

Let $\alpha $ be an arbitrary affine isometric $G$-action on a Banach space $B$ with $\delta_B(t) \ge \delta_{L^{2+\epsilon(G)}}(t)$ for all $0<t<2$ $($or $\rho_B(t) \le \rho_{L^{2+\epsilon(G)}}(t)$ for all $t>0)$. 
Define a sequence $x_n \in B$ inductively, starting from an arbitrary $x_0 \in B$. 
Given $x_n$, let $R_n = \max_{g\in K} \| \alpha(g,x_n) -x_n\|$. 
Then, applying the lemma, there exists $x_{n+1}\in B$ with $$ \| x_n - x_{n+1}\|\le CR_n$$ so that 
$$R_{n+1} = \max_{g\in K} \| \alpha(g,x_{n+1}) -x_{n+1}\| \le \frac{R_n}{2}.$$
We get $R_n \le R_0/2^n$ and 
$$\sum _{n=1}^\infty \|x_{n+1}-x_n \|_B\le CR_0 \sum _{n=1}^\infty  \frac{1}{2^n}=CR_0 < \infty.  $$
The limit of the Cauchy sequence $\{x_n\}_{n=1}^\infty$ is a $G$-fixed point of $\alpha$. 
\end{proof}

\section{Proof of Theorem \ref{thm:notFPhi}}

\begin{theoremnotFPhi}
Let $\Gamma $ be a hyperbolic group and $\mathbb{K} = \mathbb{R}$ or $\mathbb{C}$. 
Then there exists $2 \le p=p(\Gamma) < \infty$ such that,  
for any N-functions $\Phi$ and $\Psi$ satisfying 
\begin{itemize} 
\item $\Phi \in \Delta_2^{\mathbb{N}}$, and   
\item there is a constant $D>0$ and $t_0>0$ such that $\Psi (t) \le Dt^p$ for all $0< t \le t_0$, 
\end{itemize}
$\Gamma $ admits a proper affine isometric action on the $\ell^\Psi(\Gamma,\ell^\Phi(\Gamma, \mathbb{K}))$ space with gauge norm, where  
$$\ell^\Psi(\Gamma,\ell^\Phi(\Gamma,\mathbb{K})) = \left\{ \xi :\Gamma  \to \ell^{\Phi}(\Gamma, \mathbb{K}) \mid \sum_{\gamma \in \Gamma } \Psi \left(a \|\xi(\gamma)\|_{(\Phi)}\right) < \infty \text{ for some } a>0 \right\}   $$
with gauge norm 
$$\|\xi \|_{(\Psi\Phi)} = \inf \left\{ b>0 \mid \sum_{\gamma \in \Gamma } \Psi \left(\frac{\|\xi(\gamma)\|_{(\Phi)}}{b}\right)\le 1 \right\}. $$
\end{theoremnotFPhi}

\begin{proof}
Since $\Gamma $ is a countable set with counting measure, we can identify $\Gamma $ and $\mathbb{N}$ as measure spaces.  
Let $S$ be a finite generating set of $\Gamma$ and $G$ a Cayley graph of $\Gamma$ with respect to $S$.  
We endow $G$ the path metric $d$, and identify $\Gamma $ with the set of vertices of $G$. 
Let $\delta \ge 1$ be a positive integer such that all the geodesic triangles in $G$ are $\delta$-fine. 
Set $B(x,R) = \{a \in \Gamma \mid d(a,x) \le R\}$ and $S(x,R) = \{a \in \Gamma \mid d(a,x) = R\}$. 

Let $q$ be a $\Gamma $-equivalent bicombing, that is, a function assigning to each $(a,b)\in \Gamma \times \Gamma$ an oriented edge-path $q[a,b]$ from $a$ to $b$ satisfying $q[g\cdot a, g\cdot b] = g\cdot q[a,b]$ for each $a,b,g\in \Gamma $.  
Set $\mathbb{L} = \mathbb{Q}$ if $ \mathbb{K} = \mathbb{R}$ and $\mathbb{L} = \mathbb{Q}+ \sqrt{-1} \mathbb{Q}$ if $ \mathbb{K} = \mathbb{C}$.  
%$C_0(\Gamma ,\mathbb{L})$ be the space of all finitely supported 0-chains with ciefficients in $\mathbb{L}$, i.e.   
%$C_0(\Gamma ,\mathbb{L}) = \{\sum_{\gamma \in \Gamma } c_\gamma \gamma \mid c_\gamma \in \mathbb{L}\}$, 
%where $\sum_{\gamma \in \Gamma } c_\gamma \gamma$ is finitely supported in the sense that $c_\gamma =0$ for all but finitely maniy $\gamma $. 
Let $$C_0(\Gamma ,\mathbb{L}) = \left\{\sum_{\gamma \in \Gamma } c_\gamma \gamma \mid c_\gamma \in \mathbb{L},\ c_\gamma =0 \text{ for all but finitely many } \gamma \right\}$$ 
where we  identify $\Gamma $ with the standard basis of $C_0(\Gamma , \mathbb{L})$. 
Therefore the left action of $\Gamma $ on itself induces an action on $C_0(\Gamma , \mathbb{L})$. 
\begin{proposition}[\cite{MR1866802}]
There is a function $f:\Gamma \times \Gamma \to C_0(\Gamma , \mathbb{L})$ satisfies the following conditions. 
\begin{enumerate}
\item For each $a,b\in \Gamma$, $f(a,b)$ is a convex combination, i.e. its coefficients are non-negative and sum up to 1. 
\item If $d(a,b) \le 10 \delta$, then $f(b,a)=a$. 
\item If $d(a,b) \ge 10 \delta$, then $\supp f(b,a) \subset B(q[b,a](10\delta),\delta) \cap S(b,10\delta)$. 
\item f is $\Gamma$-equivariant, i.e. $f(g\cdot a ,g\cdot b) =g\cdot f(a,b)$ for any $g,a,b \in \Gamma$. 
\item There exist constants $L \ge 0$ and $0 < \lambda <1$ such that, for all $a,a',b \in \Gamma $, 
$$\|f(b,a) -f(b,a')\|_1 \le L \lambda ^{(a|a')_b},$$
where $(a|a')_b$ is the Gromov product defined by 
$$(a|a')_b=\frac{1}{2}[d(b,a)+d(b,a')-d(a,a')]$$ 
and $$\left\|\sum_{\gamma \in \Gamma } c_\gamma \gamma \right\|_1 = \sum_{\gamma \in \Gamma } |c_\gamma|$$
for $\sum_{\gamma \in \Gamma } c_\gamma \gamma \in C_0(\Gamma , \mathbb{L})$. 
\end{enumerate}
\end{proposition}
For an N-function $\Phi \in \Delta_2^{\mathbb{N}}$, endow $C_0(\Gamma , \mathbb{L})$ with the gauge-norm such as 
$$\left\|\sum_{\gamma \in \Gamma } c_\gamma \gamma \right\|_{(\Phi)} = \inf \left\{ b>0 \mid \sum_{\gamma \in \Gamma } \Phi \left(\frac{|c_\gamma|}{b}\right)\le 1 \right\} . $$
Since we can write $f(b,a)-f(b,a') = \sum_{\gamma \in B(b,10\delta ) }  c_\gamma \gamma$, we get 
\begin{eqnarray*}
&& \|f(b,a)-f(b,a')\|_{(\Phi)} \\ 
&=& \left\| \sum_{\gamma \in B(b,10\delta )  } c_\gamma \gamma\right\|_{(\Phi)}  \\ 
&=& \inf \left\{ b>0 \mid \sum_{\gamma \in B(b,10\delta ) }  \Phi \left(\frac{|c_\gamma|}{b}\right) \le 1 \right\} \\ 
&\le & \inf \left\{ b>0 \mid \sum_{\gamma \in B(b,10\delta ) }  \Phi\left(\frac{\sum_{\gamma' \in B(b,10\delta ) }|c_{\gamma'}|}{b}\right) \le 1 \right\} \\ 
%&=& |S(b,10\delta )| \inf \left\{ b \mid  \Phi\left(\frac{\| \sum_{\gamma \in S(b,10\delta )  } c_\gamma \gamma \|_1}{b}\right) \le 1 \right\} \\ 
&=&  \inf \left\{ b>0 \mid  \frac{\| \sum_{\gamma \in B(b,10\delta )  } c_\gamma \gamma \|_1}{b} \le \Phi^{-1} \left( \frac{1}{|B(b,10\delta )|} \right) \right\} \\ 
%&=& \frac{\| \sum_{\gamma \in S(b,10\delta )  } c_\gamma \gamma \|_1}{\Phi^{-1} \left( \frac{1}{|S(b,10\delta )|} \right)} \\   
&=& \frac{\| f(b,a)-f(b,a') \|_1}{\Phi^{-1} \left( \frac{1}{|B(b,10\delta )|} \right)}   .  
\end{eqnarray*}
Similarly, we can write $f(b,a) = \sum_{\gamma \in B(b,10\delta ) }  c'_\gamma \gamma$, and by the non-normalized H\"older inequality for Orlicz space, we get 
\begin{eqnarray*} 
1 
= \| f(b,a) \|_1 
\le  2\| f(b,a) \|_{(\Phi)} \| \chi_{B(b,10\delta ) }\|_{(\Phi^*)} 
= \frac{2\| f(b,a) \|_{(\Phi)} }{(\Phi^*)^{-1}(\frac{1}{|B(b,10\delta ) |})} . 
\end{eqnarray*}
As in \cite{MR2221161}, for each pair $a,b \in \Gamma $, define 
$h(b,a) = \frac{f(b,a)}{\|f(b,a)\|_{(\Phi)}}$. 
Then 
\begin{eqnarray*}
&& \left\| h(b,a)-h(b,a') \right\|_{(\Phi)} \\ 
&=& \left\|  \frac{f(b,a)}{\|f(b,a)\|_{(\Phi)}}-\frac{f(b,a')}{\|f(b,a')\|_{(\Phi)}} \right\|_{(\Phi)} \\ 
%& \le & \left\|  \frac{f(b,a)}{\|f(b,a)\|_{(\Phi)}} - \frac{f(b,a')}{\|f(b,a)\|_{(\Phi)}} \right\|_{(\Phi)}
%+ \left\|  \left( \frac{1}{\|f(b,a)\|_{(\Phi)}}-\frac{1}{\|f(b,a')\|_{(\Phi)}}\right)f(b,a') \right\|_{(\Phi)} \\  
& \le & \frac{\left\| f(b,a) - f(b,a') \right\|_{(\Phi)}}{\|f(b,a)\|_{(\Phi)}} 
+ \frac{|\|f(b,a')\|_{(\Phi)} - \|f(b,a)\|_{(\Phi)}|}{\|f(b,a)\|_{(\Phi)}} \\  
& \le & 2 \frac{\left\| f(b,a) - f(b,a') \right\|_{(\Phi)}}{\|f(b,a)\|_{(\Phi)}} \\  
&\le & \frac{ 4\| f(b,a)-f(b,a') \|_1}{\Phi^{-1} \left( \frac{1}{|B(b,10\delta )|} \right)(\Phi^*)^{-1}\left(\frac{1}{|B(b,10\delta ) |}\right)} \\ 
%&\le & \frac{ 4L \lambda ^{(a|a')_b} }{\Phi^{-1} \left( \frac{1}{|B(b,10\delta )|} \right)(\Phi^*)^{-1}\left(\frac{1}{|B(b,10\delta ) |}\right)} \\ 
& = & C \lambda ^{(a|a')_b}  
\end{eqnarray*}
where $C=\frac{ 4L }{\Phi^{-1} \left( \frac{1}{|B(b,10\delta )|} \right)(\Phi^*)^{-1}\left(\frac{1}{|B(b,10\delta ) |}\right)}$. 
Since $\Phi \in \Delta_2^{\mathbb{N}}$, the Orlicz space $\ell^{\Phi}(\Gamma,\mathbb{K})$ is the completion of $C_0(\Gamma, \mathbb{L})$ with respect to the gauge norm. 
Notice that the $\Gamma $ action on $C_0(\Gamma, \mathbb{L})$ can be extended to an isometric action on $\ell^{\Phi}(\Gamma, \mathbb{K})$. 

Let $\pi$ be a linear isometric action on $X=\ell^\Psi(\Gamma,\ell^\Phi(\Gamma,\mathbb{K}))$ defined by 
$$(\pi(g)\xi)(\gamma) = g(\xi(g^{-1}\gamma )) = (\xi(g^{-1}\gamma ))(g^{-1} \cdot)$$
for all $\xi \in X$ and $g,\gamma \in \Gamma$. 
Define a function $\eta :\Gamma \to \ell^{\Phi }(\Gamma ,\mathbb{K}) $, by $$\eta (\gamma ) = h(\gamma ,e)$$ for all $\gamma \in \Gamma$, where $e$ is the identity element in $\Gamma$.   

Let $v > 0$ such that $|B(x,r)| \le v^r$ for all $x \in \Gamma $ and $r>0$. 
Choose $p=p(\Gamma ) \ge 2$ such that $ \lambda ^{p} v<\frac{1}{2} $.  
Let $\Psi $ be a N-function such that there is a constant $D>0$ and $t_0>0$ with $\Psi (t) \le Dt^p$ for all $0< t \le t_0$. 
Since $ (\pi(g)\eta -\eta)(\gamma) = g(h(g^{-1}\gamma,e)) - h(\gamma ,e) = h(\gamma,g) - h(\gamma ,e)$ holds, we have   
\begin{eqnarray*}
\sum_{\gamma \in \Gamma } \Psi \left(\frac{\|(\pi(g)\eta -\eta)(\gamma)\|_{(\Phi)}}{c}\right)  
%&=& \inf \left\{ b>0 \mid \sum_{\gamma \in \Gamma } \Psi \left(\frac{\|g(h(g^{-1}\gamma,e)) - h(\gamma ,e)\|_{(\Phi)}}{b}\right)\le 1 \right\} \\ 
&=& \sum_{\gamma \in \Gamma } \Psi \left(\frac{\|h(\gamma,g) - h(\gamma ,e)\|_{(\Phi)}}{c}\right) \\ 
&\le & \sum_{\gamma \in \Gamma } \Psi \left(\frac{C\lambda ^{(g|e)_\gamma}  }{c}\right) \\ 
&\le & \sum_{\gamma \in \Gamma } \Psi \left(\frac{C\lambda ^{(d(\gamma,e) - d(g,e))}  }{c}\right)  \\  
&\le & \sum_{n=0}^\infty \Psi \left(\frac{C\lambda ^{(n - d(g,e))}  }{c}\right)v^n .  
\end{eqnarray*} 
For each $g\in \Gamma $, we set $n_0(g) = \min \{n\in\mathbb{N} \mid C\lambda ^{(n - d(g,e))}  \le t_0\}$,  
and $$c_0(g) =\min \left\{c \ge 1 \mid n_0(g)\Psi \left(\frac{C\lambda ^{ - d(g,e)}}{c} \right)v^{n_0(g)} \le \frac{1}{2}\right\}. $$ 
Then since $0 < \lambda <1$ and $v>1$, for $c\ge c_0(g)$ we have 
\begin{eqnarray*}
&& \sum_{n=0}^\infty \Psi \left(\frac{C\lambda ^{(n - d(g,e))}  }{c}\right)v^n  \\ 
&\le & \sum_{n=0}^{n_0(g)-1} \Psi \left(\frac{C\lambda ^{(n - d(g,e))}  }{c}\right)v^n + \sum_{n=n_0(g)}^\infty \Psi \left(\frac{C\lambda ^{(n - d(g,e))}  }{c}\right)v^n  \\ 
&\le &  n_0(g) \Psi \left(\frac{C\lambda ^{ - d(g,e)}}{c_0(g)} \right)v^{n_0(g)} + \sum_{n=n_0(g)}^\infty D\left(\frac{C\lambda ^{(n - d(g,e))}  }{c}\right)^pv^n \\ 
&\le& \frac{1}{2} +  \frac{1}{c^p}D\left(C\lambda ^{ - d(g,e)}  \right)^p \sum_{n=n_0(g)}^\infty  (\lambda^{p}v)^n  . 
\end{eqnarray*} 
Therefore 
\begin{eqnarray*}
&& \|\pi(g)\eta -\eta\|_{(\Psi\Phi)} \\ 
&=& \inf \left\{ c>0 \mid \sum_{\gamma \in \Gamma } \Psi \left(\frac{\|(\pi(g)\eta -\eta)(\gamma)\|_{(\Phi)}}{c}\right)\le 1 \right\} \\ 
&\le & \inf \left\{ c\ge c_0(g)  \mid \sum_{\gamma \in \Gamma } \Psi \left(\frac{\|(\pi(g)\eta -\eta)(\gamma)\|_{(\Phi)}}{c}\right)\le 1 \right\} \\ 
&\le& \inf \left\{ c\ge c_0(g) \mid \frac{1}{2} +  \frac{1}{c^p}D\left(C\lambda ^{ - d(g,e)}  \right)^p \sum_{n=n_0(g)}^\infty  (\lambda^{p}v)^n \le 1 \right\} \\ 
&\le & \inf \left\{ c \ge c_0(g)  \mid 2D\left(C\lambda ^{ - d(g,e)}  \right)^p  \le c^p \right\} \\ 
&=& \max\{c_0(g), (2D)^{\frac{1}{p}} C\lambda ^{ - d(g,e)} \} < \infty .  
\end{eqnarray*} 
It follows that $\pi(g) \eta -\eta $ is an element in $X$ for each $g\in \Gamma $. 
We now define an affine isometric action $\alpha $ on $X$ by $\Gamma $ by 
$$ \alpha (g) \xi = \pi(g) \xi + \pi(g) \eta - \eta $$
for all $\xi \in X$ and $g \in \Gamma $. 
For $\gamma \in q[g,e] $ with $d(\gamma ,e ) \ge 10 \delta$ and $d(\gamma ,g ) \ge 10 \delta$, 
since 
$$ B(q[\gamma ,e](10\delta ),\delta) \cap  B(q[\gamma ,g](10\delta ),\delta) = \emptyset ,  $$
we have 
$ \supp h(\gamma ,e) \cap \supp h(\gamma ,g) = \emptyset$, 
and hence  $$\| h(\gamma,g)- h(\gamma ,e)\|_{(\Phi)} \ge 1.$$ 
Thus for $g\in \Gamma $ we have
\begin{eqnarray*}
\| \pi(g)\eta-\eta\|_{(\Psi\Phi)} 
&=& \inf \left\{ c>0 \mid \sum_{\gamma \in \Gamma } \Psi \left(\frac{\|h(\gamma,g)- h(\gamma ,e)\|_{(\Phi)}}{c}\right)\le 1 \right\} \\ 
&\ge &  \inf \left\{ c>0 \mid \sum_{\gamma \in q[g,e] ; d(\gamma ,e ) \ge 10 \delta, d(\gamma ,g ) \ge 10 \delta} \Psi \left(\frac{1}{c}\right) \le 1 \right\} \\
&=& \frac{1}{\Psi^{-1}(\frac{1}{|\{\gamma \in q[g,e] ; d(\gamma ,e ) \ge 10 \delta, d(\gamma ,g ) \ge 10 \delta\}|})}\\ 
&\ge & \frac{1}{\Psi^{-1}(\frac{1}{d(g,e)-100\delta})}. 
\end{eqnarray*}
As a consequence, for every $\xi \in X$, we have 
$$\| \pi(g^{-1})\alpha (g) \xi - \xi \|_{(\Psi\Phi)}  = \| \alpha (g) \xi -\pi(g) \xi \|_{(\Psi\Phi)} = \| \pi(g)\eta-\eta\|_{(\Psi\Phi)} 
 \to \infty $$ 
as $d(g,e) \to \infty$. 
Therefore the affine isometric action $\beta (g) = \pi(g^{-1})\alpha (g) $ on $X $ of $\Gamma $ is proper. 
\end{proof}

\vspace{5mm}
\noindent 
Mamoru Tanaka,\\ 
Advanced Institute for Materials Research, Tohoku University, Sendai, 980-8577 Japan\\ 
E-mail: mamoru.tanaka@wpi-aimr.tohoku.ac.jp
\end{document}